\documentclass{article}
\usepackage{amsfonts}
\usepackage{amsmath}

\newtheorem{theorem}{Theorem}

\newtheorem{conjecture}[theorem]{Conjecture}

\newtheorem{definition}[theorem]{Definition}

\newenvironment{proof}[1][Proof]{\noindent\textbf{#1.} }{\ \rule{0.5em}{0.5em}}
\input{tcilatex}
\begin{document}

\title{Global Dynamical Behaviours and Periodicity of a Certain
Quadratic-Rational Difference Equation with Delay}
\author{Erkan Ta\c{s}demir$^{1}$, Melih G\"{o}cen$^{2}$, Y\"{u}ksel Soykan$%
^{2}$ \\
%EndAName
$^{1}$K\i rklareli University, P\i narhisar Vocational School, \\
39300, K\i rklareli, Turkey\\
Corresponding author mail: erkantasdemir@hotmail.com\\
$^{2}$Zonguldak Bulent Ecevit University, Department of Mathematics, \\
Art and Science Faculty, 67100, Zonguldak, Turkey. }
\maketitle

\textbf{Abstract}

Our aim in this paper is to deal with the dynamics of following higher order
difference equation%
\begin{equation*}
x_{n+1}=A+B\frac{x_{n-m}}{x_{n}^{2}}
\end{equation*}%
where $A,B>0$, and initial values are positive, and $m=\{1,2,...\}$\textbf{.}
Furthermore, we discuss the periodicity, boundedness, semi-cycles, global
asymptotic stability of solutions of these equations. We also handle the
rate of convergence of solutions of these difference equations.

\textbf{Keywords:} Difference equations, periodicity, boundedness,
semicycle, global asymptotic stability, rate of convergence.

\textbf{AMS Subject Classification:} 39A10, 39A23, 39A30.

\section{Introduction}

Last few decades, rational difference equations and their systems have
attracted the interest of many researchers for varied reasons. One of the
reason of this rapid growth of interest is, these equations provided a
natural description of many discrete mathematical models. Such discrete
mathematical models are often scrutinized in various fields of science and
technology for instance, biology, ecology, physiology, physics, engineering,
economics, probability theory, genetics, psychology, resource management and
population dynamics. We believe that the interest of studying rational
difference equations will increase in future years as more fascinating and
intresting results are obtaining. Although difference equations are very
simple in form, it is extremely difficult to understand thoroughly the
behaviors of their solutions. Furthermore, higher-order rational difference
equations and systems of rational equations have also been widely studied
but still have many aspects to be investigated. There are many papers
related to the rational difference equations and higher-order rational
difference equations for example, see [\cite{gocen2018}, \cite{Huristic}, 
\cite{Abualrub}].

In \cite{amleh1999}, Amleh et al. discussed the stability, boundedness and
periodic character of solutions of difference equation 
\begin{equation*}
x_{n+1}=\alpha +\frac{x_{n-1}}{x_{n}},
\end{equation*}%
where the initial values are positive numbers, and $\alpha \geq 0$.

In \cite{devault2003}, Devault et al. studied periodicity, global stability
and the boundedness of solutions of the following higher order difference
equation%
\begin{equation*}
x_{n+1}=p+\frac{x_{n-k}}{x_{n}},
\end{equation*}%
where the initial conditions are positive numbers, and $p>0$.

In \cite{saleh2005}, Saleh et al. handled the dynamical behaviours of
following higher order difference equation%
\begin{equation}
y_{n+1}=A+\frac{y_{n-k}}{y_{n}},  \label{saris526f}
\end{equation}%
with $k\in \{2,3,\cdots \}$ and $A$ is positive. The authors especially
discussed the global asymptotic stability, semi-cycle analysis and
periodicity of the unique positive equilibrium of \ Eq.(\ref{saris526f}).

In \cite{abusaris2003}, Abu-Saris et al. dealt with the global asymptotic
stability of positive equilibrium point of difference equations%
\begin{equation}
y_{n+1}=A+\frac{y_{n}}{y_{n-k}},  \label{sariseq6n5}
\end{equation}%
where $k\in \{2,3,\cdots \}$ and $A$ is positive. Moreover, in \cite%
{saleh2006b}, Saleh et al. studied the global stability of the negative
equilibrium of the difference equation (\ref{sariseq6n5}) where $k\in
\{1,2,\cdots \}$ and $A<0$.

In \cite{hamzamorsy2007}, Hamza et al. discussed the dynamics of following
difference equation%
\begin{equation}
x_{n+1}=\alpha +\frac{x_{n-1}}{x_{n}^{k}},  \label{hamza5247}
\end{equation}%
where $\alpha $, $k$ and the initial values are positive real numbers. The
authors dealt with the boundedness, oscillation behaviours and stability
analysis of unique equilibrium point of Eq.(\ref{hamza5247}).

In \cite{yalcinkaya2008}, Yal\c{c}\i nkaya handled the oscillation
behaviours, bounded solutions, periodic solutions and global stability of
solutions of difference equation 
\begin{equation*}
x_{n+1}=\alpha +\frac{x_{n-m}}{x_{n}^{k}},
\end{equation*}%
where the initial values are positive real numbers and $\alpha ,k>0$.

In \cite{Stevic2002}, Stevic investigated the dynamical properties of
difference equation%
\begin{equation*}
x_{n+1}=\frac{x_{n-1}}{g\left( x_{n}\right) },
\end{equation*}%
where $x_{-1},x_{0}>0$.

In \cite{beso2020}, Be\v{s}o et al. showed the Neimark--Sacker bifurcation,
boundedness and global attractivity of following difference equation 
\begin{equation*}
x_{n+1}=\gamma +\delta \frac{x_{n}}{x_{n-1}^{2}},
\end{equation*}%
with the initial conditions and $\gamma ,\delta $ are positive real numbers.

In \cite{Tasdemir}, Ta\c{s}demir investigated the boundedness, rate of
convergence, global asymptotic stability and periodicity of the following
higher order difference equations 
\begin{equation}
x_{n+1}=A+B\frac{x_{n}}{x_{n-m}^{2}},  \label{eqneqn1111}
\end{equation}%
where the initial conditions and $A,$ $B$ are positive real numbers and $%
m\in \left\{ 2,3,\cdots \right\} $.

Our aim in this work is to deal with the dynamics of following higher order
difference equation%
\begin{equation}
x_{n+1}=A+B\frac{x_{n-m}}{x_{n}^{2}},  \label{erk}
\end{equation}%
with $m=\{1,2,...\}$, and the initial conditions are positive numbers, and $%
A,B>0$\textbf{. }We first handle the periodicity, boundedness and
oscillation behaviors of solutions of Eq.(\ref{erk}). Moreover, we analyze
the local and global asymptotic stability of the solutions of Eq.(\ref{erk}%
). Finally, we study the rate of convergence of Eq.(\ref{erk}) and we
present some numerical examples to verify our theoretical results.

Here, we summarize the significant results and definitions on the theory of
difference equations. For more information, see \cite{Elaydi1996}, \cite%
{Kulenovic2002}, \cite{Camouzis2008} and references therein.

Let $I$ be some interval of real numbers and let $f:I^{k+1}\rightarrow I$ be
a continuously differentiable function. A difference equation of order $%
(k+1) $ is an equation of the form%
\begin{equation}
x_{n+1}=f(x_{n},x_{n-1},\cdots ,x_{n-k}),\text{ \ \ \ }n=0,1,\cdots .
\label{ana}
\end{equation}

A solution of Eq.(\ref{ana}) is a sequence $\{x_{n}\}_{n=-k}^{\infty }$ that
satisfies Eq.(\ref{ana}) for all $n\geq -k$.

Suppose that the function $f$ is continuously differentiable in some open
neighborhood of an equilibrium point $\overline{x}.$ Let%
\begin{equation*}
q_{i}=\frac{\partial f}{\partial u_{i}}(\overline{x},\overline{x},\cdots ,%
\overline{x}),\text{ \ for }i=0,1,\cdots ,k
\end{equation*}%
denote the partial derivative of $f(u_{0},u_{1},\cdots ,u_{k})$ with respect
to $u_{i}$ evaluated at the equilibrium point $\overline{x}$ of Eq.(\ref{ana}%
)

\begin{definition}
The equation%
\begin{equation}
z_{n+1}=q_{0}z_{n}+q_{1}z_{n-1}+\cdots +q_{k}z_{n-k},\text{ }n=0,1,\cdots ,
\label{abc}
\end{equation}%
is called the linearized equation of Eq.(\ref{ana}) about the equilibrium
point $\overline{x}$.
\end{definition}

\begin{theorem}[Clark's Theorem]
\label{ann} Consider Eq.(\ref{abc}). Then, 
\begin{equation*}
\sum_{i=0}^{k}\left\vert q_{i}\right\vert <1.
\end{equation*}%
is a sufficient condition for the locally asymptotically stability of Eq.(%
\ref{ana}).
\end{theorem}

\ Consider the scalar $k$th-order linear difference equation%
\begin{equation}
x_{n+k}+p_{1}(n)x_{n+k-1}+\cdots +p_{k}(n)x_{n}=0,  \label{rateofconv5ncy}
\end{equation}%
where $k$ is a positive integer and $p_{i}:\mathbb{Z}^{+}\rightarrow \mathbb{%
C}$ for $i=1,\cdots ,k$. Assume that%
\begin{equation}
q_{i}=\underset{k\rightarrow \infty }{\lim }p_{i}(n),i=1,\cdots ,k,
\label{rateoflimit}
\end{equation}%
exist in $\mathbb{C}$. Consider the limiting equation of (\ref%
{rateofconv5ncy}):%
\begin{equation}
x_{n+k}+q_{1}x_{n+k-1}+\cdots +q_{k}x_{n}=0.  \label{limitingequa6gn6}
\end{equation}

\begin{theorem}[Poincar\'{e}'s Theorem]
Consider (\ref{rateofconv5ncy}) subject to condition (\ref{rateoflimit}).
Let $\lambda _{1},\cdots ,\lambda _{k}$ be the roots of the characteristic
equation%
\begin{equation}
\lambda ^{k}+q_{1}\lambda ^{k-1}+\cdots +q_{k}=0  \label{poincarevn7dnr6}
\end{equation}%
of the limiting equation (\ref{limitingequa6gn6}) and suppose that $%
\left\vert \lambda _{i}\right\vert \neq \left\vert \lambda _{j}\right\vert $
for $i\neq j$. If $x_{n}$ is a solution of (\ref{rateofconv5ncy}), then
either $x_{n}=0$ for all large $n$ or there exists an index $j\in \{1,\cdots
,k\}$ such that%
\begin{equation*}
\underset{n\rightarrow \infty }{\lim }\frac{x_{n+1}}{x_{n}}=\lambda _{j}.
\end{equation*}
\end{theorem}

The following results were obtained by Perron, and one of Perron's results
was improved by Pituk, see \cite{pituk2002}.

\begin{theorem}
\label{rateofpoinc54bf5}Suppose that (\ref{rateoflimit}) holds. If $x_{n}$
is a solution of (\ref{rateofconv5ncy}), then either $x_{n}=0$ eventually or%
\begin{equation*}
\underset{n\rightarrow \infty }{\lim }\sup \left( \left\vert x_{j}\left(
n\right) \right\vert \right) ^{1/n}=\lambda _{j}.
\end{equation*}%
where $\lambda _{1},\cdots ,\lambda _{k}$ are the (not necessarily distinct)
roots of the characteristic equation (\ref{poincarevn7dnr6}).
\end{theorem}

\begin{theorem}[See \protect\cite{bilgin2017}]
\label{thebound}Let $n\in N_{n_{0}}^{+}$ and $g\left( n,u,v\right) $ be a
nondecreasing function in $u$ and $v$ for any fixed $n$. Suppose that, for $%
n\geq n_{0}$, the inequalities%
\begin{eqnarray*}
y_{n+1} &\leq &g\left( n,y_{n},y_{n-1}\right) , \\
u_{n+1} &\geq &g\left( n,u_{n},u_{n-1}\right)
\end{eqnarray*}%
hold. Then%
\begin{eqnarray*}
y_{n_{0}-1} &\leq &u_{n_{0}-1}, \\
y_{n_{0}} &\leq &u_{n_{0}}
\end{eqnarray*}%
implies that 
\begin{equation*}
y_{n}\leq u_{n},n\geq n_{0}.
\end{equation*}
\end{theorem}

Firstly, we are in a position to study the dynamics of the higher order
difference equation (\ref{erk}). The Eq.(\ref{erk}) which by the change of
variables%
\begin{equation*}
y_{n}=\frac{x_{n}}{A},
\end{equation*}%
reduces to the following difference equation%
\begin{equation}
y_{n+1}=1+p\frac{y_{n-m}}{y_{n}^{2}},  \label{eqn2222}
\end{equation}%
where $p=\frac{B}{A^{2}}$. Henceforth, we consider the difference equation (%
\ref{eqn2222}). Note that Eq.(\ref{eqn2222}) has an unique positive
equilibrium point such that 
\begin{equation*}
\bar{y}=\frac{1+\sqrt{1+4p}}{2}.
\end{equation*}

\section{Boundedness of Solutions of Eq.(\protect\ref{eqn2222})}

Now, we handle the bounded solutions of Eq.(\ref{eqn2222}). We also find out
that Eq.(\ref{eqn2222}) has bounded solutions.

\begin{theorem}
\label{boundedsol712}Let $0<p<1$. Then, there exists the bounded solutions
of Eq.(\ref{eqn2222}) such as%
\begin{equation*}
1<y_{n}\leq \frac{1}{1-p}+\sqrt[m+1]{p^{n}}\sum_{j=1}^{m+1}c_{j}\left( e^{%
\frac{2\pi i}{m+1}\left( j-1\right) }\right) ^{n},
\end{equation*}%
where $c_{j},$ $j=1,2,\cdots ,m+1$, are arbitrary constants, and $n\in
\left\{ 0,1,2,\cdots \right\} $, and $\sqrt[m+1]{p}$ is one of the $(m+1)$th
roots of $p$.
\end{theorem}

\begin{proof}
Let $p>0$, and $\left\{ y_{n}\right\} _{n=-m}^{\infty }$ be a positive
solution of Eq.(\ref{eqn2222}). Then, we obtain from Eq.(\ref{eqn2222})%
\begin{eqnarray*}
y_{1} &=&1+p\frac{y_{-m}}{y_{0}^{2}}>1, \\
y_{2} &=&1+p\frac{y_{1-m}}{y_{1}^{2}}>1.
\end{eqnarray*}%
Hence, by induction, we get $y_{n}>1$ for $n\geq 1$.

Now, we take care of the other side. We have from Eq.(\ref{eqn2222}) 
\begin{equation*}
y_{n+1}=1+p\frac{y_{n-m}}{y_{n}^{2}}\leq 1+py_{n-m}.
\end{equation*}%
According to Theorem \ref{thebound}, we consider a sequence $\left\{
u_{n}\right\} _{n=0}^{\infty }$, and $y_{n}\leq u_{n},$ $n=0,1,\cdots ,$ and 
\begin{equation}
u_{n+1}=1+pu_{n-m},n\geq 1,  \label{unequat84ng}
\end{equation}%
such that 
\begin{equation}
u_{s+i}=y_{s+i},s\in \left\{ -m,-m+1,\cdots \right\} ,i=\left\{ 0,1,2,\cdots
\right\} ,n\geq s.  \label{abdtrne4c2}
\end{equation}%
The characteristic polynomial to Eq.(\ref{unequat84ng}) is 
\begin{equation*}
P_{m+1}\left( \lambda \right) =\lambda ^{m+1}-p.
\end{equation*}%
Thus, we have the roots of characteristic polynomial as follows: 
\begin{equation*}
\lambda _{j}=\sqrt[m+1]{p}e^{\frac{2\pi i}{m+1}\left( j-1\right) },
\end{equation*}%
where $j=1,2,\cdots ,m+1.$ The homogeneous solution of Eq.(\ref{unequat84ng}%
) is 
\begin{equation*}
u_{h}=\sum_{j=1}^{m+1}c_{j}\sqrt[m+1]{p^{n}}\left( e^{\frac{2\pi i}{m+1}%
\left( j-1\right) }\right) ^{n},
\end{equation*}%
where $c_{j}$ are arbitrary constants for $j=1,2,\cdots ,m+1$. Now, we
handle the equilibrium solution of Eq.(\ref{unequat84ng}). From Eq.(\ref%
{unequat84ng}), we get that 
\begin{equation*}
\bar{u}=\frac{1}{1-p}.
\end{equation*}%
Therefore, the solution of Eq.(\ref{unequat84ng}) is 
\begin{equation}
u_{n}=\frac{1}{1-p}+\sum_{j=1}^{m+1}c_{j}\sqrt[m+1]{p^{n}}\left( e^{\frac{%
2\pi i}{m+1}\left( j-1\right) }\right) ^{n},  \label{solu7vnt}
\end{equation}%
where $c_{j}$ are arbitrary constants for $j=1,2,\cdots ,m+1$. Furthermore,
we have from (\ref{abdtrne4c2}) and (\ref{solu7vnt})%
\begin{equation*}
y_{n+1}-u_{n+1}\leq p\left( y_{n}-u_{n}\right) ,
\end{equation*}%
where $n>s,$ and $p\in \left( 0,1\right) $. Therefore, we get $y_{n}\leq
u_{n},$ $n>s$. So, the proof completed.
\end{proof}

\section{Oscillation Behaviors of Eq.(\protect\ref{eqn2222})}

In this section, we discuss the semi-cycles of Eq.(\ref{eqn2222}). We also
reveal the oscillation behaviours of solutions of Eq.(\ref{eqn2222}) in
detail.

\begin{theorem}
Let $\{y_{n}\}_{n=-m}^{\infty }$ be a positive solution of Eq.(\ref{eqn2222}%
). Then, the following statements are true:

(i) The every semi-cycle at most $m$ terms.

(ii) Every solution of Eq.(\ref{eqn2222}) oscillates about the positive
equilibrium $\bar{y}$.
\end{theorem}

\begin{proof}
We firstly handle the positive semi-cycle of solution of Eq.(\ref{eqn2222}).
The negative semi-cycle is similar and can be omitted. Assume that Eq.(\ref%
{eqn2222}) has a positive semi-cycle with $m$ terms. Suppose that $y_{N}$ is
the first term in this positive semi-cycle. Therefore, we get 
\begin{equation*}
y_{N},y_{N+1},\cdots ,y_{N+m-1}>\bar{y}.
\end{equation*}%
Hence, we obtain from Eq.(\ref{eqn2222})%
\begin{equation*}
y_{N+m}=1+p\frac{y_{N-1}}{y_{N+m-1}^{2}}<1+p\frac{y_{N-1}}{\bar{y}^{2}}<\bar{%
y}.
\end{equation*}%
So, we have that a semi-cycle consists at most $m$ terms. We also get that
every solution of Eq.(\ref{eqn2222}) oscillates about $\bar{y}$. The proof
completed as desired.
\end{proof}

\begin{theorem}
Let $m$ be an odd number and let 
\begin{equation}
y_{-m},y_{-m+2},\cdots ,y_{-1}\leq \bar{y}\text{ and }y_{-m+1},y_{-m+3},%
\cdots ,y_{0}>\bar{y}.  \label{semicycle7gn6}
\end{equation}%
Then, every semi-cycle of Eq.(\ref{eqn2222}) has lenght one. Additionally,
the solution $\{y_{n}\}_{n=-m}^{\infty }$ of Eq.(\ref{eqn2222}) is
oscillatory about unique positive equilibrium point $\bar{y}$.
\end{theorem}

\begin{proof}
Let $\{y_{n}\}_{n=-m}^{\infty }$ be a positive solution of Eq.(\ref{eqn2222}%
). Assume that (\ref{semicycle7gn6}) holds. Hence, we obtain from Eq.(\ref%
{eqn2222}),%
\begin{eqnarray*}
y_{1} &=&1+p\frac{y_{-m}}{y_{0}^{2}}<\bar{y}, \\
y_{2} &=&1+p\frac{y_{-m+1}}{y_{1}^{2}}>\bar{y}, \\
y_{3} &=&1+p\frac{y_{-m+2}}{y_{2}^{2}}<\bar{y}.
\end{eqnarray*}%
Therefore, we have by induction%
\begin{equation*}
y_{2n+1}=1+p\frac{y_{2n-m}}{y_{2n}^{2}}<\bar{y},
\end{equation*}%
and 
\begin{equation*}
y_{2n}=1+p\frac{y_{2n-(m+1)}}{y_{2n-1}^{2}}>\bar{y}.
\end{equation*}
\end{proof}

\section{Periodicity of Eq.(\protect\ref{eqn2222})}

Now, we study the existence of periodic solutions of Eq.(\ref{eqn2222}).

\begin{theorem}
Assume that $m$ is an even number. Then, Eq.(\ref{eqn2222}) has no two
periodic solution.
\end{theorem}

\begin{proof}
Let $m$ is an even number. We also suppose that Eq.(\ref{eqn2222}) has two
periodic solution such that 
\begin{equation*}
\cdots ,\alpha ,\beta ,\alpha ,\beta ,\cdots
\end{equation*}%
where $\alpha $ and $\beta $ are positive numbers and $\alpha \neq \beta $.
Hence, we get the followings%
\begin{eqnarray}
y_{2n+1} &=&1+p\frac{y_{2n-m}}{y_{2n}^{2}}\Rightarrow \alpha =1+\frac{p}{%
\beta },  \label{eqobm7per2} \\
y_{2n+2} &=&1+p\frac{y_{2n+1-m}}{y_{2n+1}^{2}}\Rightarrow \beta =1+\frac{p}{%
\alpha }.  \notag
\end{eqnarray}%
where $n\geq 1$. Therefore, we obtain that%
\begin{equation*}
\left( \alpha -\beta \right) \left( 1-\frac{p}{\alpha \beta }\right) =0.
\end{equation*}%
Now, from our supposition, we have \ $\alpha \neq \beta $. Thus, we get%
\begin{equation*}
1-\frac{p}{\alpha \beta }=0\Rightarrow p=\alpha \beta .
\end{equation*}%
From (\ref{eqobm7per2}), we have 
\begin{equation*}
\alpha \beta =\beta +p\Rightarrow \beta =0\text{.}
\end{equation*}
So, we have a contradiction. The proof completed as desired.
\end{proof}

\section{Stability of Solutions of Eq.(\protect\ref{eqn2222})}

In this section, we deal with the asymptotic stability of the solutions of
Eq.(\ref{erk}). \ Firstly, we find the linearized equation associated with
Eq.(\ref{eqn2222}) about its positive equilibrium point.

Let $I$ be some interval of real numbers and let 
\begin{equation*}
f:I^{m+1}\rightarrow I,
\end{equation*}%
be a continuously differentiable function such that $f$ is defined by 
\begin{equation*}
f\left( y_{n},y_{n-1},\cdots ,y_{n-m}\right) =1+p\frac{y_{n-m}}{y_{n}^{2}}%
\text{.}
\end{equation*}%
Thus, we obtain that%
\begin{equation*}
q_{0}=\frac{\partial f}{\partial y_{n}}=-\frac{2p}{\bar{y}^{2}},
\end{equation*}

\begin{equation*}
q_{1}=q_{2}=\cdots =q_{m-1}=0,
\end{equation*}

\begin{equation*}
q_{m}=\frac{\partial f}{\partial y_{n-m}}=\frac{p}{\bar{y}^{2}}.
\end{equation*}%
Then the linearized equation of Eq.(\ref{eqn2222}) about its unique positive
equilibrium point $\bar{y}$ is:%
\begin{equation}
z_{n+1}+\frac{2p}{\bar{y}^{2}}z_{n}-\frac{p}{\bar{y}^{2}}z_{n-m}=0.
\label{linearizedeqlastngy4}
\end{equation}%
Hence, the characteristic equation of Eq.(\ref{eqn2222}) is as follows, 
\begin{equation}
\lambda ^{m+1}+\frac{2p}{\bar{y}^{2}}\lambda ^{m}-\frac{p}{\bar{y}^{2}}=0.
\label{characteris5mby68d}
\end{equation}

\begin{theorem}
\label{localstablethe}The positive equilibirium $\overline{y}$ of Eq.(\ref%
{eqn2222}) is locally asymptotically stable when $p\in \left( 0,\frac{3}{4}%
\right) .$
\end{theorem}

\begin{proof}
From Theorem \ref{ann}, all roots of the characteristic equation of Eq.(\ref%
{linearizedeqlastngy4}) lie in an open disc $\left\vert \lambda \right\vert
<1$, if%
\begin{equation*}
\left\vert q_{0}\right\vert +\left\vert q_{1}\right\vert +\left\vert
q_{2}\right\vert +\cdots \left\vert q_{m}\right\vert <1\text{.}
\end{equation*}%
It follows from (\ref{characteris5mby68d}) that 
\begin{equation*}
\left\vert q_{0}\right\vert +\left\vert q_{1}\right\vert +\left\vert
q_{2}\right\vert +\cdots \left\vert q_{m}\right\vert =\frac{3p}{\bar{y}^{2}}.
\end{equation*}%
Note that 
\begin{equation*}
\frac{p}{\bar{y}^{2}}=\frac{2p+1-\sqrt{4p+1}}{2p}.
\end{equation*}%
With many numerical calculations, we get that 
\begin{eqnarray*}
\left\vert q_{0}\right\vert +\left\vert q_{1}\right\vert +\left\vert
q_{2}\right\vert +\cdots \left\vert q_{m}\right\vert &=&\frac{3p}{\bar{y}^{2}%
}<1, \\
\frac{3\left( 2p+1-\sqrt{4p+1}\right) }{2p} &<&1, \\
\frac{4p+3-3\sqrt{4p+1}}{2p} &<&0.
\end{eqnarray*}%
Hence, we obtain from $p>0$, 
\begin{equation*}
\left( \sqrt{4p+1}-1\right) \left( \sqrt{4p+1}-2\right) <0.
\end{equation*}%
Therefore, we get that $0<p<\frac{3}{4}$. And the proof is complete.
\end{proof}

\begin{theorem}
The equilibrium point $\bar{y}$ of Eq.(\ref{eqn2222}) is globally
asymptotically stable if $0<p<\frac{1}{2}$.
\end{theorem}

\begin{proof}
From Theorem \ref{boundedsol712}, we know that there exist $I$ and $S$ such
that 
\begin{equation*}
1<I=\underset{n\rightarrow \infty }{\lim }\inf y_{n}\leq S=\underset{%
n\rightarrow \infty }{\lim }\sup y_{n}.
\end{equation*}%
Hence, we get from Eq.(\ref{eqn2222})%
\begin{equation*}
I\geq 1+p\frac{I}{S^{2}}\text{ and }S\leq 1+p\frac{S}{I^{2}}.
\end{equation*}%
Therefore, we obtain that 
\begin{equation*}
S+p\frac{I}{S}\leq IS\leq I+p\frac{S}{I}.
\end{equation*}%
Thus, we have 
\begin{equation*}
\left( S-I\right) \left( 1-p\left( \frac{1}{S}+\frac{1}{I}\right) \right)
\leq 0.
\end{equation*}%
From $S\geq I>1$ and $0<p<\frac{1}{2}$, we also get 
\begin{equation*}
1-p\left( \frac{1}{S}+\frac{1}{I}\right) >0.
\end{equation*}%
So, we have $S\leq I$ which the result follows. Hence, the equilibrium point 
$\bar{y}$ of Eq.(\ref{eqn2222}) is globally asymptotically stable if $0<p<%
\frac{1}{2}.$
\end{proof}

\begin{conjecture}
Many numerical simulations show that If $\frac{1}{2}\leq p<\frac{3}{4}$,
then the equilibrium point $\bar{y}$ of Eq.(\ref{eqn2222}) is globally
asymptotically stable.
\end{conjecture}

\section{Rate of Convergence of Eq.(\protect\ref{eqn2222})}

Here, we investigate the rate of convergence of solutions of Eq.(\ref%
{eqn2222}).

\begin{theorem}
Let $\lambda _{j}$ be roots of characteristic equation (\ref%
{characteris5mby68d}) where $j\in \left\{ 1,\cdots ,k\right\} $. Then, every
solution of Eq.(\ref{eqn2222}) ensures the following relations:%
\begin{equation*}
\underset{n\rightarrow \infty }{\lim }\left\vert \frac{y_{n+1}-\bar{y}}{%
y_{n}-\bar{y}}\right\vert =\left\vert \lambda _{j}\right\vert ,
\end{equation*}%
and 
\begin{equation*}
\underset{n\rightarrow \infty }{\lim }\sup \left( \left\vert y_{n}-\bar{y}%
\right\vert \right) ^{1/n}=\left\vert \lambda _{j}\right\vert .
\end{equation*}
\end{theorem}

\begin{proof}
According to Eq.(\ref{eqn2222}), we have that 
\begin{eqnarray*}
y_{n+1}-\bar{y} &=&\left( 1+p\frac{y_{n-m}}{y_{n}^{2}}\right) -\left( 1+p%
\frac{\bar{y}}{\bar{y}^{2}}\right) \\
&=&-\frac{p\left( y_{n}+\bar{y}\right) }{\bar{y}y_{n}^{2}}\left( y_{n}-\bar{y%
}\right) +\frac{p}{y_{n}^{2}}\left( y_{n-m}-\bar{y}\right) .
\end{eqnarray*}%
Now, we consider $e_{n}=y_{n}-\bar{y}$. Then, we obtain 
\begin{equation*}
e_{n+1}+p_{n}e_{n}+q_{n}e_{n-m}=0,
\end{equation*}%
such that 
\begin{equation*}
p_{n}=-\frac{p\left( y_{n}+\bar{y}\right) }{\bar{y}y_{n}^{2}},
\end{equation*}%
and 
\begin{equation*}
q_{n}=\frac{p}{y_{n}^{2}}.
\end{equation*}%
Therefore, we get from globally asymptotic stability%
\begin{equation*}
\underset{n\rightarrow \infty }{\lim }p_{n}=-\frac{2p}{\bar{y}^{2}},
\end{equation*}%
and 
\begin{equation*}
\underset{n\rightarrow \infty }{\lim }q_{n}=\frac{p}{\bar{y}^{2}}.
\end{equation*}%
So, the proof is completed.
\end{proof}

\end{document}